\documentclass[12pt]{amsart}
\usepackage{amssymb}

\theoremstyle{plain}
\newtheorem{theorem}{Theorem}
\newtheorem*{Theorem}{Theorem 13'}
\newtheorem{proposition}[theorem]{Proposition}
\newtheorem{lemma}[theorem]{Lemma}

\theoremstyle{definition}
\newtheorem{remark}{Remark}

\DeclareMathOperator{\End}{End}
\DeclareMathOperator{\Aut}{Aut}

\DeclareMathOperator{\ssupp}{ssupp}

\renewcommand{\geq}{\geqslant}
\renewcommand{\leq}{\leqslant}

\begin{document}
\title[Centering probability measures]{Centering problems for
 probability\\ measures on finite dimensional\\ vector spaces}
%\dedicatory{}
\author[Andrzej \L uczak]{Andrzej \L uczak}
\address{Faculty of Mathematics and Computer Science\\
        \L\'od\'z University\\
        ul.~ S. Banacha 22\\
        90-238~ \L\'od\'z, Poland}
\email{anluczak@math.uni.lodz.pl}
\thanks{Work supported by KBN grant 2 1020 91 01}
\date{}
\begin{abstract}
 The paper deals with various centering problems for probability
 measures on finite dimensional vector spaces. We show that for every
 such measure there exists a vector $h$ satisfying\newline
 $\mu*\delta(h)=S(\mu*\delta (h))$ for each
 symmetry $S$ of $\mu$, generalizing thus Jurek's result obtained
 for full measures. An explicit form of the $h$ is given for
 infinitely divisible $\mu$. The main result of the paper consists
 in the analysis of quasi-decomposable (operator-semistable and
 operator-stable) measures and finding conditions for the
 existence of a `universal centering' of such a measure to a
 strictly quasi-decomposable one.
\end{abstract}
\maketitle

\section*{Introduction}
The general setup for the problems considered in this paper may be
formulated as follows. For a finite dimensional vector space $V$, we
are given a class $\Phi$ of transformations defined on a subset
$\mathcal{S}$ of all probability measures $\mathcal{P}(V)$ on $V$ and
taking values in $\mathcal{P}(V)$. Let $\mu$ be a measure belonging
to $\mathcal{S}$, and denote by $\Phi_0(\mu)$ a subset of $\Phi$
consisting of the elements $\varphi$ having the property
\[
 \mu=\varphi(\mu)*\delta(h_{\varphi}),
\]
with some $h_{\varphi}\in V$. We are looking for a `universal
centering' of $\mu$ with respect to $\Phi_0(\mu)$, by which is
meant an element $h^{\prime}\in V$, independent of
$\varphi\in\Phi_0(\mu)$, such that for all
$\varphi\in\Phi_0(\mu)$ we have
\[
 \mu*\delta(h^{\prime})=\varphi(\mu*\delta(h^{\prime}).
\]
Two cases are dealt with in detail:
\begin{enumerate}
 \item[1.] $\mathcal{S}=\mathcal{P}(V),\,
 \Phi=\End V\text{ and }\varphi(\mu)=\mu\circ\varphi^{-1}$.
 \item[2.] $\mathcal{S}$ --- infinitely divisible measures,
 $\Phi=\{\varphi=(a,A): a\in(0,\infty),\newline A\in
 \End V\},\text{ and }\varphi(\mu)=(A\mu)^{1/a}$.
\end{enumerate}
The first case was considered by Z. Jurek in \cite{J2} for $S$ being
the set of full measures, so that $\Phi_0(\mu)$ is the so-called
symmetry group of $\mu$; as for the second, note that in order that
$\Phi_0(\mu)$ be nontrivial $\mu$ must be $(a, A)$-quasi-decomposable
with some $a\ne 1$ and $A\in\End V$, i.e.,
\[
 \mu^a = A\mu *\delta(h),\tag{$\ast$}
\]
and our problem consists in centering $\mu$ to a strictly
quasi-decomposable measure, that is we look for an $h'\in V$ such
that
\[
 (\mu * \delta(h'))^a = A(\mu * \delta(h'))
\]
for all pairs $(a, A)$ satisfying ($\ast$). In the case when ($\ast$)
is satisfied by pairs $(t, t^B)$ for all $t >0$, i.e., when $\mu$ is
operator-stable, a partial question concerning only existence and not
universality has been solved in \cite{Sa}. However, also in this
case, our solution of the general problem is given in a form which
appears to be well suited to both (operator-stable as well as
operator-semistable) possible situations and is considerably
different in form and method from that of \cite{Sa}.

The paper bears a direct connection to the theory of operator-limit
distributions on finite dimensional vector spaces. A useful source of
information about this theory is monograph \cite{JM} to which the
reader is referred for additional facts, explanations, comments etc.

\section{Preliminaries and notation}
Throughout the paper, $V$ will stands for a finite dimensional real
vector space with an inner product $(\cdot, \cdot)$ yielding a norm
$\|\cdot\|$, and the $\sigma$-algebra $\mathcal{B}(V)$ of its Borel
subsets. We let $\End V$ denote the set of all linear operators on
$V$, whereas $\Aut V$ stands for the linear invertible operators.

Let $A\colon V\to W$ be a linear mapping into a finite dimensional
real vector space $W$, and let $\mu$ be a (probability) measure over
$(V,\,\mathcal{B}(V))$. The measure $A\mu$ on $(W,\,\mathcal{B}(W))$
is defined by
\[
 A\mu(E) = \mu(A^{-1}(E)),\qquad E\in\mathcal{B}(W).
\]
In particular, if $\xi\colon\Omega\to V$ is a random variable taking
values in $V$ and $\mu$ is the law of $\xi$, then $A\mu$ is the law
of $A\xi$.

The following equalities are easily verified
\[
 A(B\mu) = (AB)\mu,\quad \widehat{A\mu}(v) =\hat\mu(A^*v),
 \quad A(\mu * \nu) = A\mu * A\nu,
\]
for linear operators $A,B$ and probability measures $\mu,\nu$ (here
$\hat{}$ denotes the characteristic function, and the asterisk $*$
stands for the convolution of measures or for the adjoint of an
operator, as the case may be). By $\delta(h)$ we denote the
probability measure concentrated at point $h$.

A probability measure on $V$ is called \emph{full} if it is not
concentrated on any proper hyperplane of $V$. Let $\mu$ be a
probability measure on $V$. Then there exists a smallest hyperplane
$U$ of $V$ such that $\mu$ is concentrated on $U$ and, by a little
abuse of language, we can speak of $\mu$ being full on $U$. In this
case, there is the unique subspace $W$ of $V$ and an element $h\in V$
such that $U = W +h$. We call $W$ the \emph{supporting subspace} of
$\mu$ and denote it by $W = \ssupp(\mu)$. It is clear that
$\ssupp(\mu) = \{0\}$ if and only if $\mu =\delta(h)$ for some $h\in
V$, and $\ssupp(\mu) = V$ if and only if $\mu$ is full.

A linear operator  $S$ on $V$ is called a \emph{symmetry} of $\mu$ if
there is an $h\in V$ such that $\mu = S\mu * \delta(h)$. The set of
all symmetries of $\mu$ is denoted by $\mathbb{A}(\mu)$. Let us
recall that if $\mu$ is full, then $\mathbb{A}(\mu)$ is a compact
subgroup of $\Aut V$ (cf. \cite[Corollary 2.3.2]{JM} or \cite{Sh,U}).
We recall that an infinitely divisible measure $\mu$ on $V$ has the
unique representation as a triple $[m, D, M]$, where $m\in V$, $D$ is
a non-negative linear operator on $V$, and $M$ is  the L\'evy
spectral measure of $\mu$, i.e. a Borel measure defined on $V_0 = V
-\{0\}$ such that $\int_{V_0} \|v\|^2/(1 + \|v\|^2)\,M(dv)<\infty$.
The characteristic function of $\mu$ has then the form
\[
 \hat \mu(u) =\exp\bigg\{i(m, u) -\frac{1}{2}(Du, u) +\int_{V_0}
 \Big(e^{i(v, u)} -1 -\frac{i(v, u)}{1+\|v\|^2}\Big)\,M(dv)\bigg\}
\]
(cf. e.g. \cite{P}). A straightforward calculation shows that
for\linebreak $\mu =[m, D, M]$ and $A\in\End V$, we have $A\mu =[m',
ADA^*, AM]$, where
\begin{equation}\label{1}
 m' = Am +\int_{V_0}\frac{\|u\|^2-\|Au\|^2}{(1+\|A u\|^2)
 (1+\|u\|^2)}Au \,M(du).
\end{equation}

One of the main objects considered in this paper is the class of
operator-semistable (operator-stable) or, more generally,
quasi-de\-com\-po\-sa\-ble measures. A measure $\mu$ on $V$ is called
$(a, A)$-\emph{quasi-decomposable} with $a >0$, $a\ne 1$, $A\in\End
V$, if it is infinitely divisible and
\begin{equation}\label{2}
 \mu^a = A\mu *\delta(h_{a, A})\qquad\text{for some}\quad h_{a, A}\in V.
\end{equation}
If $h_{a, A} =0$, then $\mu$ is called \emph{strictly} $(a,
A)$-\emph{quasi-decomposable}. $\mu$ is called
\emph{quasi-decomposable} if it is $(a, A)$-quasi-decomposable for
some pair $(a, A)$. It is known that a quasi-decomposable measure is
operator semistable, i.e. arises as the limit law of a sequence
\[
 A_n\nu^{k_n} *\delta(h_n),
\]
where $\nu\in\mathcal{P}(V)$, $A_n\in \End V$, $h_n\in V$,
$k_{n+1}/k_n\to r\geq 1$ and the power $\nu^{k_n}$ is taken in the
sense of convolution, and the converse is true if the limit measure
is full (cf. \cite{HM,J,J1,K,L1,Sh} for a more detailed description
of this class).

\section{Universal centering with respect to symmetries}
In this section we show that for any probability measure $\mu$ on $V$
there exists an $h'\in V$ such that, for each $S\in\mathbb{A}(\mu)$,
\[
 \mu * \delta(h') = S(\mu * \delta(h')).
\]
As noted in the Introduction, this problem was solved in \cite{J2}
under the fullness assumption on $\mu$.

In addition to the general solution, we give an explicit form of the
$h'$ for $\mu$ being infinitely divisible.

Our first lemma is a slight refinement of Proposition 1 from \cite{Sh}.

\begin{lemma}\label{L1}
Let $W$ be a subspace of $V$ and denote by $W^{\perp}$ its orthogonal
subspace. Then $|\hat{\mu}(v)| = 1$ for $v\in W^{\perp}$ if and only
if $\mu =\nu *\delta(h)$ with $\nu(W) =1$ and $h\in W^{\perp}$.
\end{lemma}
\begin{proof}
Assume that $|\hat{\mu}(v)| =1$ for $v\in W^{\perp}$ and let $P$ be
the orthogonal projection on $W^{\perp}$. Then $P\mu$ is a measure
concentrated on $W^{\perp}$ and for any $v\in V$, we have
\[
 |\widehat{P\mu}(v)| = |\hat{\mu}(P v)| =1,
\]
thus $P\mu =\delta(h)$ for some $h\in W^{\perp}$, which gives the
equality
\[
 \mu = P \mu * \mu * \delta(-h).
\]
On account of \cite[Proposition 1.5]{U} or \cite[Theorem
2.3.6(b)]{JM} we have
\[
 \mu = P\mu * (I - P)\mu,
\]
where $I$ is the identity operator, and putting $\nu = (I - P)\mu$,
we get the formula $\mu =\nu * \delta(h)$ with $\nu$ concentrated on
$W$.

Conversely, if $\mu =\nu *\delta(h)$ with $\nu$ concentrated on $W$,
then $\hat{\nu}(v) =1$ for $v \in W^\perp$, and so $|\hat\mu(v)| =
|\hat\nu(v)| =1$ for $v\in W^{\perp}$.
\end{proof}

\begin{lemma}\label{L2}
Let $W$ be a subspace of $V$. Assume that $\mu$ is concentrated on
$W$ and the decomposition $\mu = \nu * \lambda$ holds. Then $\nu
=\nu_1 * \delta(h)$, $\lambda =\lambda_1 * \delta(-h)$, where $\nu_1$
and $\lambda_1$ are concentrated on $W$, and $h\in W^{\perp}$.
\end{lemma}
\begin{proof}
For $v\in W^{\perp}$ we have
\[
 1 = |\hat{\mu}(v)| = |\hat{\nu}(v)| |\hat{\lambda}(v)|,
\]
and thus
\[
 |\hat{\nu}(v)| = |\hat\lambda(v)| =1.
\]
By Lemma \ref{L1}, $\nu =\nu_1 * \delta(h)$, $\lambda =\lambda_1 *
\delta(h')$ for some $h, h'\in W^{\perp}$, where $\nu_1$ and
$\lambda_1$ are concentrated on $W$. Moreover,
\[
 \mu =\mu * \lambda =\nu_1 * \lambda_1 * \delta(h+h'),
\]
which yields $h +h'\in W$, so $h +h' =0$ and the assertion follows.
\end{proof}

The next proposition gives an important property of the supporting
subspaces in the case of a decomposition of measures.

\begin{proposition}\label{P3}
Let $A\in\End V$, and let $\mu, \nu, \lambda$ be probability measures
on $V$ such that
\[
 \mu = A\nu * \lambda.
\]
Put $W = \ssupp(\mu)$, $U = \ssupp(\nu)$. Then $A(U) \subset W$.
\end{proposition}
\begin{proof}
There is an $h\in V$ such that $\nu$ is full on $U +h$. We have
\[
 \mu * \delta(-h) = A(\nu * \delta(-h)) * \lambda * \delta(Ah -h),
\]
and putting
\[
 \mu' = \mu * \delta(-h),\quad \nu' = \nu * \delta(-h),\quad
 \lambda' = \lambda * \delta(A h - h)
\]
we get
\[
 \mu' = A\nu' * \lambda',
\]
moreover, the measure $\nu'$ is full on $U$. We claim that $A\nu'$ is
full on $A(U)$. Indeed, let $X$ be a subspace of $A(U)$ and let $x_0
= A u_0$ be an element in $A(U)$ such that $A\nu'$ is concentrated on
$X + x_0$. Then
\[
 1 = A\nu'(X +x_0) =\nu'(A^{-1}(X +x_0)) =\nu'(A^{-1}(X) + u_0),
\]
and the fullness of $\nu'$ on $U$ yields $U \subset A^{-1}(X)$, thus
\[
 X\subset A(U) \subset AA^{-1}(X) = X,
\]
showing that $X = A(U)$ and, consequently, $A\nu'$ is full on $A(U)$.
On account of Lemma \ref{L2}, there is a $v_0\in W^{\perp}$ such that
\[
 (A\nu' * \delta(v_0))(W)=1,
\]
and thus
\[
 A\nu'(W - v_0) =1,
\]
consequently,
\[
 A\nu'(A(U)\cap (W - v_0)) =1.
\]
But $A(U)\cap (W -v_0)$ is a hyperplane in $A(U)$, so the fullness of
$A\nu'$ on $A(U)$ implies that
\[
 A(U)\cap (W -v_0) = A(U).
\]
Hence $A(U) \subset W - v_0$, which shows that $v_0\in W$, and
finally, \linebreak $A(U) \subset W$, finishing the proof.
\end{proof}

As an easy consequence of the above proposition and Jurek's result we
get the following theorem on the existence of universal centering
with respect to $\mathbb{A}(\mu)$ for any probability measure $\mu$
on $V$.

\begin{theorem}\label{T4}
Let $\mu$ be a probability measure on $V$. Then there exists $h'\in
V$ such that for each $S\in\mathbb{A}(\mu)$
\[
 \mu * \delta(h') = S(\mu * \delta(h')).
\]
\end{theorem}
\begin{proof}
Let $W = \ssupp(\mu)$. Choose $h_0$ such that the measure \linebreak
$\mu' = \mu * \delta(h_0)$ is concentrated (and full) on $W$. We have
\[
 \mathbb{A}(\mu') =\{S\in \End V: \mu' = S\mu' * \delta(h)\;
 \text{for some } h\in V\},
\]
and Proposition \ref{P3} yields that $\mu'$ is concentrated on $W$
which implies that the $h$ occurring in the definition of
$\mathbb{A}(\mu')$ must be in $W$. Consider $\mu'$ only on the
subspace $W$. Then, again by virtue of Proposition \ref{P3}, we have
\[
 \mathbb{A}(\mu'|_W) = \{S|_W: S\in \mathbb{A}(\mu')\}.
\]
Since $\mu'$ is full on $W$, we infer, on account of \cite{J2}, that
there exists $h''\in W$ such that for each $S\in\mathbb{A}(\mu')$
\[
 \mu' * \delta(h'') = (S|_W)(\mu' * \delta(h'')).
\]
Now
\[
 \mu' * \delta(h'') = \mu * \delta(h_0) * \delta(h'') =
 \mu * \delta(h_0 +h'')
\]
and
\begin{align*}
 (S|_W)(\mu' *\delta(h'')) &= S\mu' * \delta(S h'') =
 S(\mu *\delta(h_0)) * \delta(S h'')\\&= S(\mu * \delta(h_0 +h'')).
\end{align*}
Putting
\[
 h' = h_0 +h'',
\]
we obtain thus
\[
 \mu *\delta(h') = S(\mu * \delta(h'))
\]
for each $S\in\mathbb{A}(\mu')$, and since clearly $\mathbb{A}(\mu')
=\mathbb{A}(\mu)$, the conclusion follows.
\end{proof}

Now we shall find the form of a universal centering for any
infinitely divisible measure. Let $\mu = [m, D, M]$ be such a
measure, and assume first that $\mu$ is full, $m=0$, and
$\mathbb{A}(\mu)$ is a subgroup of the orthogonal group $\mathbb{O}$
in $V$. For each $S\in\mathbb{A}(\mu)$ we then have $S\mu = [m',
SDS^*, SM] = [m', D, M]$ where, by virtue of \eqref{1},
\[
 m' = \int_{V_0}\frac{\|u\|^2-\|Su\|^2}{(1+ \|Su\|^2)
 (1 +\|u\|^2)}Su\,M(du) =0,
\]
since $S$ is an isometry. Thus $S\mu =\mu$, i.e. any measure $\mu =
[0, D, M]$ having the property $\mathbb{A}(\mu) \subset\mathbb{O}$ is
itself universally centered. Now let us assume only that $\mu =[m, D,
M]$ is full. Since $\mathbb{A}(\mu)$ is compact, there exists an
invertible operator $T$ on $V$ such that
\begin{equation}\label{3}
T\mathbb{A}(\mu) T^{-1}\subset\Bbb O.
\end{equation}
It is easily seen that
\begin{equation}\label{4}
 T\mathbb{A}(\mu) T^{-1} =\mathbb{A}(T\mu).
\end{equation}
The measure $T\mu$ has the form $T\mu = [m', TDT^*, TM]$ with
\begin{equation}\label{5}
 m' = Tm + \int_{V_0}\frac{\|u\|^2-\|Tu\|^2}{(1+\|T u\|^2)(1+\|u\|^2)}Tu\,M(du).
\end{equation}
According to the first part of our considerations, $-m'$ is a
universal centering for $T\mu$. Let $S\in\mathbb{A}(\mu)$. Then
$TST^{-1}\in\mathbb{A}(T \mu)$, and we have
\[
 T \mu *\delta(-m') = TST^{-1}(T\mu * \delta(-m')) =
 TS(\mu * \delta(-T^{-1} m')),
\]
which yields the equality
\[
 \mu *\delta(-T^{-1} m') = S(\mu * \delta(-T^{-1} m')),
\]
meaning that $h' = T^{-1} m'$ is a universal centering for $\mu$.
From \eqref{5} we get the formula
\[
 h' = -\bigg(m+ \int_{V_0}\frac{\|u\|^2-\|Tu\|^2}{(1+\|T u\|^2)
 (1+\|u\|^2)}\,M(du)\bigg).
\]

Finally, let $\mu = [m, D, M]$ be an arbitrary infinitely divisible
measure on $V$. Put $W = \ssupp(\mu)$, and let $h_0\in V$ be such
that $\mu' = \mu *\delta(h_0) = [m+h_0, D, M]$ is full on $W$. The
preceding discussion applied to the measure $\mu'$ on $W$ shows that
a universal centering $h''$ for $\mu'$ has the form
\[
 h'' = -\bigg(m +h_0 +\int_{W_0}\frac{\|u\|^2-\|Tu\|^2}{(1+\| Tu\|^2)
 (1+\|u\|^2)}u\,M(du)\bigg),
\]
where $T$ is an invertible operator on $W$ such that $T \mathbb
{A}(\mu')T^{-1}$ is a subgroup of the isometries on $W$. It is clear
that $h' = h'' + h_0$ is a universal centering for $\mu$, so for this
centering we have the formula
\[
 h' = -\bigg(m + \int_{W_0}\frac{\|u\|^2-\|Tu\|^2}{(1 +\|T u\|^2)
 (1+\|u\|^2)}u\,M(du)\bigg),
\]
where $W = \ssupp(\mu)$, and $T$ is an invertible operator on $W$
such that $T(\mathbb{A}(\mu)|\,W)T^{-1}$ is a subgroup of the
isometries on $W$.

\section{Centering problem for quasi-decomposable measures}
For an infinitely divisible measure $\mu$ on $V$ and $a >0$ put,
following \cite{Sh} (cf. also \cite[p. 187]{JM}),
\[
 G_a(\mu) = \{A\in\End V: \mu^a = A\mu * \delta(h)\quad\text{for some }
 h\in V\}.
\]
We recall that $\mu$ is quasi-decomposable if $G_a(\mu)\ne\emptyset$
for some $a\ne 1$. In this section, given a quasi-decomposable
measure $\mu$, we aim at finding conditions for the existence of an
$\hat{h}\in V$ such that for any $a >0$ with $G_a(\mu)\ne\emptyset$
and any $A\in G_a(\mu)$ the following equality holds
\begin{equation}\label{6}
 (\mu * \delta(\hat{h}))^a = A(\mu * \delta(\hat{h})).
\end{equation}
$\mu$ is then said to have a \emph{universal quasi-decomposability
centering}.

So, let us assume that $G_a(\mu)\ne\emptyset$ for some $a\ne 1$,
i.e., that for $\mu$ equality \eqref{2} holds. Thus if we have
\eqref{6} with some $\hat{h}\in V$, then
\[
 \mu^a *\delta(a \hat{h}) = A\mu * \delta(A \hat{h}),
\]
yielding, by \eqref{2}, the equality
\[
 A\mu * \delta(h_{a, A} + a\hat{h}) = A\mu *\delta(A\hat{h}),
\]
which means that
\begin{equation}\label{7}
 h_{a, A} = A\hat{h} - a\hat{h}.
\end{equation}
On the other hand, it is immediately seen that \eqref{7} implies
\eqref{6} under the assumption of the $(a, A)$-quasi-decomposability
of $\mu$, so for such $\mu$ we have equivalence of \eqref{6} and
\eqref{7}. Thus our task consists in finding conditions for the
existence of a solution $\hat{h}$ of equation \eqref{7} and showing
that this solution is independent of $a$ and $A$.

First we address the problem of the universality of centering. This
will be performed in two steps. In the first one we shall show that
if, for a given $a$ for which formula \eqref{2} holds, there is an
$\hat{h}_0$ satisfying \eqref{6} (or \eqref{7}) for some $A_0\in
G_a(\mu)$, then there is an $\hat{h}$ satisfying \eqref{6} for all
$A\in G_a(\mu)$. In the second step, we prove that the existence of
centering for some $a$ yields the existence of centering for all the
$a$ that can occur in formula \eqref{2}, thus that this centering is
universal.

In the first part of our considerations we may assume, in view of
Proposition \ref{P3} and the obvious fact that the existence of
universal quasi-decomposability centering is not affected by shifts,
that $\mu$ is full. We then have

\begin{lemma}\label{L5}
Assume that $\mu$ is full. Then for any $A\in G_a(\mu)$ the mappings
$G_a(\mu)\ni B\mapsto B^{-1}A$ and $G_a(\mu)\ni B\mapsto AB^{-1}$ are
bijections from $G_a(\mu)$ onto $\mathbb{A}(\mu)$.
\end{lemma}
\begin{proof}
It is easily seen that $B^{-1}\in G_{1/a}(\mu)$ for $B\in G_a(\mu)$,
and thus
\[
 \mu^{1/a} = B^{-1}\mu * \delta(h_{1/a, B^{-1}}),
\]
giving the equalities
\begin{align*}
 \mu&= B^{-1} \mu^a * \delta(ah_{1/a, B^{-1}}) = B^{-1}
 (A\mu * \delta(h_{a, A})) * \delta(a h_{1/a, B^{-1}})\\
 &= B^{-1} A\mu * \delta(B^{-1} h_{a, A} + ah_{1/a, B^{-1}}),
\end{align*}
which shows that $B^{-1} A\in\mathbb{A}(\mu)$. For any
$S\in\mathbb{A}(\mu)$, $A\in G_a(\mu)$, we have
$S^{-1}\in\mathbb{A}(\mu)$, so the operator $B = A S^{-1}$ belongs to
$G_a(\mu)$ and
\[
 S = B^{-1}A,
\]
showing that the mapping $B\mapsto B^{-1}A$ is onto
$\mathbb{A}(\mu)$. Since it is injective the conclusion follows.
Analogously we deal with the case of the mapping $B\mapsto AB^{-1}$.
\end{proof}

The above mentioned fact that the existence of universal
quasi-decomposability centering is not affected by shifts allows us
to assume further that $\mu$ is universally centered with respect to
$\mathbb{A}(\mu)$. This assumption is made in the remainder of the
paper.

\begin{lemma}\label{L6}
For any $A, B\in G_a(\mu)$ we have $A\mu = B\mu$ and $h_{a, A} =
h_{a, B}$.
\end{lemma}
\begin{proof}
The following equality holds
\[
 A\mu * \delta(h_{a, A}) = B\mu * \delta(h_{a, B}),
\]
which gives
\[
 B^{-1} A\mu * \delta(B^{-1} h_{a, A}) = \mu * \delta(B^{-1}h_{a, B}),
\]
that is
\[
 \mu = B^{-1} A\mu *\delta(B^{-1}(h_{a, A} - h_{a, B})).
\]
Since $B^{-1} A\in\mathbb{A}(\mu)$ and $\mu$ is universally centered
with respect to $\mathbb{A}(\mu)$, we get
\[
 B^{-1}(h_{a, A} - h_{a, B}) =0,
\]
consequently, $h_{a, A} = h_{a, B}$ and $A\mu = B\mu$.
\end{proof}

The lemma above says that, with $\mu$ universally centered with
respect to $\mathbb{A}(\mu)$, we have the equality
\[
 \mu^a = A\mu * \delta(h_a),\qquad A\in G_a(\mu),
\]
with the same $h_a$ for all $A\in G_a(\mu)$. This yields an important
property of the $h_a$.

\begin{lemma}\label{L7}
For each $S\in\mathbb{A}(\mu)$, we have $Sh_a = h_a$.
\end{lemma}
\begin{proof}
We have, for $S\in\mathbb{A}(\mu)$,
\[
 \mu^a = S\mu^a = S A \mu * \delta(S h_a),
\]
moreover, since by Lemma \ref{5}, $S A\in G_a(\mu)$, it follows that
\[
 \mu^a = S A\mu *\delta(h_a),
\]
which proves the claim.
\end{proof}

Finally, let us make our last simplification. For $T\in\Aut V$ we
clearly have
\[
 G_a(T\mu) = TG_a(\mu)T^{-1},
\]
so $\mu$ is $(a, A)$-quasi-decomposable if and only if $T\mu$ is $(a,
TAT^{-1})$-quasi-decomposable; moreover, equality \eqref{6} is
equivalent to the equality
\[
 (T\mu * \delta(T \hat{h}))^a = TAT^{-1}(T \mu *\delta(T\hat{h})).
\]
Therefore $\hat{h}$ is a universal quasi-decomposability centering of
$\mu$ if and only if $T\hat{h}$ is a universal quasi-decomposability
centering of $T\mu$. Now taking $T$ such that \eqref{3} and \eqref{4}
hold, the above considerations allow us to assume that
$\mathbb{A}(\mu)\subset\mathbb{O}$. Let
\[
 W =\{v: Sv = v, \quad S\in\mathbb{A}(\mu)\}
\]
be the fixed-point space for $\mathbb{A}(\mu)$, and let $P$ be the
orthogonal projection onto $W$.

\begin{proposition}\label{P8}
For each $A\in G_a(\mu)$, we have $AP = PA$.
\end{proposition}
\begin{proof}
Take arbitrary $A, B\in G_a(\mu)$. Since $B^{-1}A\in
\mathbb{A}(\mu)$, we get for each $v\in W$
\[
 B^{-1} Av =v,
\]
giving
\begin{equation}\label{8}
 Av = Bv.
\end{equation}
For any $S\in\mathbb{A}(\mu)$ we have $SA\in G_a(\mu)$, thus if $v\in
W$, then
\[
 S Av = Av,
\]
that is
\[
 A(W) \subset W,
\]
or, equivalently,
\begin{equation}\label{9}
 PAP = AP.
\end{equation}

Now put $S = AB^{-1}$. Then $S\in \mathbb{A}(\mu)$, and since
$\mathbb{A}(\mu)$ is a subgroup of the orthogonal group, we get
\[
 S^{-1} = S^* = B^{*-1} A^*\in\mathbb{A}(\mu),
\]
which, as in the first part of the proof, yields
\[
 A^*v = B^*v,\quad v\in W.
\]
For any $S\in\mathbb{A}(\mu)$, we have
\[
 (S A^*)^* = AS^* = AS^{-1}\in G_a(\mu),
\]
and hence
\[
 S A^*v = A^*v\quad v\in W,
\]
giving the equality
\[
 PA^*P = A^*P.
\]
Upon taking adjoints, we obtain
\[
 PAP = PA,
\]
which, together with \eqref{9}, gives the desired result.
\end{proof}

Now we are in a position to prove the universality of centering with
respect to $G_a(\mu)$, under the assumption of the existence of a
centering for an operator from $G_a(\mu).$

\begin{proposition}\label{P9}
Assume that for some $A_0\in G_a(\mu)$ there is an $\hat{h}_0$ such
that
\begin{equation}\label{10}
 (\mu * \delta(\hat{h}_0))^a = A_0(\mu * \delta(\hat{h}_0)).
\end{equation}
Then there exists $\hat{h}$ such that for all $A\in G_a(\mu)$
equality \eqref{6} holds. Moreover, $\hat{h}$ is also a universal
centering with respect to $\mathbb{A}(\mu).$
\end{proposition}
\begin{proof}
As we have shown before, equality \eqref{10} is equivalent to the
equality
\[
 h_a = A_0\hat{h}_0 - a\hat{h}_0
\]
and as $h_a\in W$ by Lemma \ref{7}, we get
\[
 h_a = P h_a = P A_0 \hat{h}_0 - a P\hat{h}_0
 =A_0 P\hat{h}_0 - a P \hat{h}_0.
\]
Putting
\[
 \hat{h} = P \hat{h}_0,
\]
we obtain
\[
 h_a = A_0\hat{h} - a\hat{h},
\]
moreover, since $\hat h\in W$, we have by \eqref{8}
\[
 A_0\hat{h} = A\hat{h}
\]
for all $A\in G_a(\mu)$, which leads to the equality
\[
 h_a = A\hat{h} - a\hat{h},\qquad A\in G_a(\mu),
\]
proving the first part of the claim. The second part follows from the
first and Lemma \ref{L5}.
\end{proof}

For our further analysis, it will be convenient to rewrite condition
\eqref{7} in a slightly different form. Let $T\in\End V$ and let
$\mathcal{N}(T)$ denote its null space, i.e.
\[
 \mathcal{N}(T) =\{v\in V: Tv =0\}.
\]
From elementary Hilbert space theory and the finite dimensionality of
$V$, we have the  following orthogonal decomposition
\begin{equation}\label{11}
 V = \mathcal{N}(T^*) \oplus T(V).
\end{equation}
Now condition \eqref{7} means simply that $h_{a, A}\in (A - aI)(V)$,
which by \eqref{11} is equivalent to
\begin{equation}\label{12}
 h_{a, A}\perp \mathcal{N}(A^* - aI),
\end{equation}
which is the form we shall employ.

Now we shall analyze the universality with respect to various $a$'s
that can occur in formula \eqref{2}. According to \cite[Theorem
3.2]{L2} there are two possibilities: either
\begin{enumerate}
 \item[(i)] $a = c^n$ for a unique $0 < c <1$ and some integer $n$,
\end{enumerate}
or
\begin{enumerate}
 \item[(ii)] $a$ may be an arbitrary positive real number, in which
 case for $\mu$ the following formula holds
\begin{equation}\label{13}
 \mu^t = t^B \mu * \delta(h_{t, B}),\qquad t >0
\end{equation}
\end{enumerate}
for some $B \in\End V$ and $t^B$ defined as $t^B = e^{(\log t)B}$,
that is, $\mu$ is operator-stable. We shall call these two cases
discrete and continuous, respectively, and shall deal with them
separately.

\emph{Discrete case}. According to our previous considerations we may
assume that $\mu$ is centered universally with respect to
$\mathbb{A}(\mu)$. Then
\[
 \mu^c = A\mu * \delta(h_c)\quad\text{for each}\quad A\in G_c(\mu),
\]
and iterating the equality above, we obtain
\[
 \mu^{c^n} = A^n \mu * \delta(c^{n-1} h_c + c^{n-2} Ah_c +\dots + A^{n-1}h_c)
\]
and
\[
 \mu^{(1/c)^n} = A^{-n} \mu *\delta((1/c)^{n-1} h_{1/c} +
 (1/c)^{n-2} A^{-1} h_{1/c} +\dots + (A^{-1})^{n-1} h_{1/c})
\]
for all positive integers $n$. Denoting
\[
 h_n= h_{c^n},\quad p_n(b, A)= b^{n-1}I + b^{n-2}A +\dots
 + A^{n-1},\quad n=1,2,\dots,
\]
we get the formulas
\[
 h_n = p_n(c, A)h_1,\quad h_{-n} =p_n(c^{-1}, A^{-1})h_{-1},
 \qquad n =1,2,\dots;
\]
moreover, it is immediately seen that
\begin{equation}\label{14}
 h_1 =-c Ah_{-1}.
\end{equation}

Now we are in a position to set the problem of the universality of
centering together with an important point on its existence.

\begin{proposition}\label{P10}
There exists a universal quasi-decomposability centering for $\mu$ if
and only if for some integer $n$ and $A_n\in G_{c^n}(\mu)$ there
exists a centering of $\mu$ with respect to the pair $(c^n, A_n)$.
\end{proposition}
\begin{proof}
Assume that $n$ is positive and that $\mu$ may be centered with
respect to the pair $(c^n, A_n)$ with some $A_n\in G_{c^n}(\mu)$. On
account of Proposition \ref{P9} we may assume that this centering is
universal with respect to the whole of $G_{c^n}(\mu)$. Take an
arbitrary $A\in G_c(\mu)$. Then $A^n\in G_{c^n}(\mu)$ and since
\[
 \mu^{c^n} = A^n \mu * \delta(h_n),
\]
the existence of a centering for $(c^n, A_n)$ yields the condition
\[
 h_n\perp \mathcal{N}(A^{*n} - c^n I).
\]
For each $v\in\mathcal{N}(A^* - cI)$ we have
\[
 p_n(c, A)^*v = p_n(c, A^*)v = c^{n-1}v +c^{n-2}A^*v+\dots
 + A^{*n-1}v = n c^{n-1}v,
\]
and since
\[
 \mathcal{N}(A^* -cI)\subset \mathcal{N}(A^{*n} -c^n I),
\]
we get
\begin{align*}
 0 &=(h_n,v) = (p_n(c, A)h_1, v) = (h_1, p_n(c A)^*v)\\
 &=n c^{n-1}(h_1, v),\qquad v\in\mathcal{N}(A^* - cI),
\end{align*}
which means that
\begin{equation}\label{15}
 h_1\perp \mathcal{N}(A^* - cI).
\end{equation}

For $n$ negative, we would obtain, considering the pair $(c^{-1},
A^{-1})$ instead of $(c, A)$, the condition
\[
 h_{-1} \perp \mathcal{N}(A^{-1*}- c^{-1} I),
\]
which, by \eqref{14}, gives again condition \eqref{15}. But this
condition together with Proposition \ref{P9} say that there is a
centering $\hat{h}$  universal with respect to $G_a(\mu)$. Thus
\[
 (\mu * \delta(\hat{h}))^c = A(\mu * \delta(\hat{h}))\qquad
 \text{for each}\quad A\in G_a(\mu),
\]
and, consequently,
\[
 (\mu * \delta(\hat h))^{c^n} = A^n(\mu * \delta(\hat h))
 \quad n =0, \pm 1,\dots .
\]
For any $A_n\in G_{c^n}(\mu)$, we have $A_n = A^n S$ with some $S\in
\mathbb{A}(\mu)$, and since $\hat{h}$ is also universal with respect
to $\mathbb{A}(\mu)$, we have
\[
 A_n(\mu *\delta(\hat{h})) = A^n S(\mu *\delta(\hat{h}))
 = A^n(\mu *\delta(\hat{h})) = (\mu *\delta(\hat{h}))^{c^n},
\]
showing the universality of centering.
\end{proof}

\emph{Continuous case}. First, notice that formulas \eqref{1} and
\eqref{2} lead to the following equality for the shift for $\mu = [m,
D, M]$
\begin{equation}\label{16}
 h_{a, A} = am - Am - \int_{V_0}\frac{\|u\|^2-\|Au\|^2}{(1+\|A u\|^2)
 (1+\|u\|^2)}Au\,M(du),
\end{equation}
which  for $\mu$ satisfying \eqref{13} takes the form
\[
 h_{t, B} = tm - t^Bm - \int_{V_0}\frac{\|u\|^2-\|t^Bu\|^2}
 {(1 +\|t^B u\|^2)(1 +\|u\|^2)}t^Bu\,M(du).
\]
Put, for the sake of convenience,
\[
 f_B(t)=h_{e^t,B},\qquad t\in\mathbb{R}.
\]
Then we have for $f_B$
\begin{equation}\label{17}
 f_B(t) = e^tm - e^{tB}m + \int_{V_0} \frac{\|e^{tB}u\|^2
 - \|u\|^2}{(1+\|u\|^2)(1+\|e^{tB}u\|^2)}e^{tB}u\,M(du).
\end{equation}
For each fixed $u\in V$, consider the function
\[
 g(t) = \|e^{tB}u\|^2,\qquad t\in\mathbb{R}.
\]
We have
\begin{equation}\label{18}
 g'(t) = 2(B e^{tB} u, e^{tB}u),\qquad t\in\mathbb{R},
\end{equation}
and since $e^{tB} u\to u$ as $t\to 0$, we get for sufficiently small
$t$'s
\[
 \bigg|\frac{\|e^{tB}u\|^2 - \|u\|^2}{t}\bigg| \leq C\|B\|\|u\|^2
\]
which gives the following estimation
\[
 \bigg\|\frac{1}{t} \frac{\|e^{tB} u\|^2 - \|u\|^2}
 {(1+ \|u\|^2)(1+\|e^{tB} u\|^2)}u\bigg\| \leq
 \frac{C\|B\|\|u\|^3}{(1 + \|u\|^2)(1+ \frac{1}{2}\|u\|^2)}.
\]
But the function on the right-hand side is $M$-integrable, thus by
Le\-besgue's  theorem we may pass to the limit with $t\to 0$ under
the integral sign in the following expression
\begin{align*}
 \frac{f_B(t)}{t} &= -\frac{e^{tB} -e^tI}{t} m\\&+ \int_{V_0}
 \frac{\|e^{tB}u\|^2 - \|u\|^2}{t} \frac{e^{tB}u}{(1 +\|e^{tB} u\|^2)
 (1 + \|u\|^2)}\,M(du)
\end{align*}
and obtain, taking into account \eqref{18},
\begin{equation}\label{19}
 v_0 =\lim_{t\to 0} \frac{f_B(t)}{t} = (I -B)m +\int_{V_0}
 \frac{2(B u, u)}{(1+ \|u\|^2)^2}u\,M(du).
\end{equation}
Since $f_B(0) =0$, we have
\[
 v_0 = f'_B(0).
\]
Now, according to \cite[Formula (8.2), p. 64]{Sh} or \cite[Sec.4.9,
p. 236]{JM}, the function $h$ satisfies the following equation
\begin{equation}\label{20}
 h_{st, B} = t^B h_{s,B} + s h_{t, B},\qquad s,t >0,
\end{equation}
which implies that for $f_B$ we have
\begin{equation}\label{21}
 f_B(s +t) = e^{tB} f_B(s) + e^s f_B(t),\qquad s,t\in\mathbb {R}.
\end{equation}
In \cite{Sh} (cf. also \cite{JM}) equation (20) is solved in general,
under the assumption $1\notin\operatorname{sp}\,B$. We shall find the
form of the function $f_B$ without any restrictions on the spectrum
(however, it should be kept in mind that we do have the existence of
$f'_B(0)$ at our disposal).

\begin{lemma}\label{L11}
The function $f_B$ has the  form
\begin{equation}\label{22}
 f_B(t) = e^t \int_0^t e^{s(B-I)}v_0\,ds,
\end{equation}
where $v_0 = f'_B(0)$ is given by equality \eqref{19}.
\end{lemma}
\begin{proof}
For each fixed $t$ and any $s$ we have
\[
 \frac{f_B(t+s) - f_B(t)}{s} = e^{tB} \frac{f_B(s)}{s}
 + \frac{e^s-1}{s} f_B(t),
\]
and passing to the limit with $s\to 0$ yields the equation
\begin{equation}\label{23}
 f'_B(t) = e^{tB} v_0 +f_B(t).
\end{equation}
It follows from e.g. \cite[Chapter 10, p. 169]{B} that the general
solution of \eqref{23} has the form
\[
 f_B(t) = e^t u_0 +\int_0^t e^{t-s} e^{sB}v_0\,ds,
\]
and taking into account our initial condition $f_B(0) =0$ we get
\eqref{22}.
\end{proof}

The next proposition sets the problem of the universality of
centering; it also adds an important point in the question of
existence.

\begin{proposition}\label{P12}
The following conditions are equivalent:
\begin{enumerate}
\item[(i)] $v_0 \perp \mathcal{N}(B^* - I)$;

\item[(ii)] there exists a universal centering;

\item[(iii)] there exists a centering for some $t' >0$,\, $t'\ne1$.
\end{enumerate}
\end{proposition}
\begin{proof}
(i) $\Longrightarrow$ (ii) By virtue of decomposition \eqref{11}, we
have
\[
 v_0 = (B - I) v_1\qquad\text{for some}\quad v_1\in V,
\]
and accordingly
\[
 f_B(t) = e^t \int_0^t e^{s(B-I)}(B - I)v_1\,ds.
\]
But the function $s\mapsto e^{s(B-I)}(B - I)v_1$ under the integral
sign is the derivative of the function $s\mapsto e^{s(B- I)}v_1$, so
we get
\[
 f_B(t) = e^t[e^{t(B- I)}v_1 -v_1] = e^{tB} v_1- e^tv_1,
\]
that is
\[
 h_{t, B} = t^B v_1 - t v_1,
\]
which means that $v_1$ is a universal centering.

(ii) $\Longrightarrow$ (iii) Obvious.

(iii) $\Longrightarrow$ (i) Consider the operator-valued function
\[
 H(t) = \int_0^t e^{s(B- I)}\,ds,\qquad t\in\mathbb{R}.
\]
Then
\[
 f_B(t) = e^t H(t) v_0,
\]
and for any $v\in \mathcal{N}(B^* - I)$, $s\in\mathbb{R}$,
\[
 e^{s(B^*-I)}v =v.
\]
Accordingly, for such $v$'s
\[
 H(t)^*v = \int_0^t e^{s(B^*-I)}v\,ds = tv,
\]
which gives

\begin{equation}\label{24}
 (f_B(t), v) = (e^t H(t)v_0, v) = e^t(v_0, H(t)^* v) = t e^t(v_0, v).
\end{equation}
Now let a centering for some $t' >0$, $t\ne 1$ be given. This means
that
\[
 h_{t', B} = t'^{B} v_1 -t' v_1
\]
for some $v_1\in V$, or with $t_0 = \log t' \ne 0,$
\[
 f_B(t_0) = e^{t_0B} v_1 - e^{t_0} v_1 = (e^{t_0 B} - e^{t_0} I) v_1.
\]
The last equality is, on account of decomposition \eqref{11},
equivalent to the relation
\[
 f_B(t_0) \perp \mathcal{N}(e^{t_0 B^*} - e^{t_0} I),
\]
and since
\[
 \mathcal{N}(B^* - I) \subset \mathcal{N}(e^{t_0 B^*} - e^{t_0} I),
\]
we get
\[
 f_B(t_0) \perp \mathcal{N}(B^* - I).
\]
Taking into account \eqref{24} we obtain for each
$v\in\mathcal{N}(B^* - I)$
\[
 0 = (f_B(t_0), v) = t_0 e^{t_0}(v_0, v),
\]
so $(v_0, v) =0$ and $v_0 \perp \mathcal{N}(B^* - I)$.
\end{proof}

What we are left with now is the existence problem. Again, as in the
analysis of universality, it will be useful to distinguish the
discrete and continuous cases, although, as we shall see, there is a
remarkable similarity between them.

For a more detailed analysis we shall need a description of the
L\'evy measure $M$, which can be found in \cite{L1} (discrete case)
and in \cite{HM,J1,JM} (continuous case). To keep this paper as
self-contained as possible, we describe below the main points.

\emph{Discrete case}. Considering, if necessary, $1/a$ instead of $a$
we may assume that $a <1$, and further that $\|A\| <1$. Then, putting
\[
 Z_A = \{v: \|v\| \leq 1\quad\text{and}\quad \|A^{-1}v\| >1\},
\]
we have the following representation for $M$

\begin{equation}\label{25}
 M(E) =\sum_{n=-\infty}^\infty a^{-n} M(A^{-n} E\cap Z_A),\quad E
 \text{ --- Borel subset of } V_0,
\end{equation}
i.e., $M$ is determined by its restriction to $Z_A$ which, in turn,
may be an (almost) arbitrary finite Borel measure (see Remark
\ref{R1} below). Formula \eqref{25} can be rewritten in the form
\begin{equation}\label{26}
 M(E) = \sum_{n=-\infty}^\infty a^{-n} \int_{Z_A}
 \boldsymbol{1}_E(A^n u)\,M(du),
\end{equation}
and for any $M$-integrable function $f$ on $V_0$ we have
\begin{equation}\label{27}
 \int_{V_0} f(u)\,M(du) =\sum_{n=-\infty}^\infty a^{-n}
 \int_{Z_A} f(A^n u)\,M(du).
\end{equation}

\begin{remark}\label{R1}
The only restriction to the arbitrariness of $M|\,Z_A$ lies in the
fact that, in general, the measure $M$ is concentrated on some
subspace of $V$ determined by eigenvalues of $A$. More precisely, if
we put \linebreak $W = \ssupp(\mu)$, then we have a decomposition
\[
 W = X\oplus Y,
\]
where $X$ and $Y$ are $A$-invariant and such that
\begin{align*}
 \operatorname{sp}(A|\,X) &\subset \{z\in\mathbb{C}: |z|^2 < a\},\\
 \operatorname{sp}(A|\,Y) &\subset \{z\in\mathbb{C}: |z|^2 = a\};
\end{align*}
and $M$ must be concentrated on $X$ (see \cite{J,L1} for details).
Similar remarks apply to the measure $K_B$ in the continuous case
below.
\end{remark}

\emph{Continuous case}. Put
\[
 L_B = \{v: \|v\| =1\quad\text{and}\quad \|t^B v\| >1\quad
 \text{for}\quad t >1\},
\]
and define the mixing measure $K_B$ on the Borel subsets $E$ of $L_B$
by
\[
 K_B(E) = M(\{t^B v: v\in E, \quad t\geq 1\}).
\]
Then we have the following continuous counterpart of \eqref{26}
\begin{equation}\label{28}
 M(E) =\int_{L_B} \int_0^\infty {\boldsymbol{1}}_E(t^B u)\,
 \frac{dt}{t^2}\,K_B(du),
\end{equation}
and for any $M$-integrable function $f$ on $V_0$
\[
 \int_{V_0} f(u)\,M(du) = \int_{L_B} \int_0^\infty f(t^B u)\,
 \frac{dt}{t^2}\,K_B(du).
\]
Substituting in the last formula $t$ in place of $\log t$, we get
\begin{equation}\label{29}
 \int_{V_0} f(u)\,M(du) = \int_{L_B} \int_{-\infty}^\infty
 f(e^{tB} u) e^{-t}\,dt\,K_B(du).
\end{equation}

Now, we can formulate our final result.

\begin{theorem}\label{T13}
Let $\mu$ be a quasi-decomposable measure. Put
\[
 W_1 = \{v: A^* v = av\} = \mathcal{N}(A^* - aI),
\]
if for $\mu$ the discrete case holds, and
\[
 W_2 = \{v: B^* v =v\} =\mathcal{N}(B^* - I),
\]
if for $\mu$ the continuous case holds. Denote in these two cases
\[
 N_1 = Z_A,\qquad \nu_1 = M|\,Z_A,\qquad N_2 = L_B,\qquad \nu_2 = K_B.
\]
Then there is a universal quasi-decomposability centering
for $\mu$ if and only if
\begin{equation}\label{30}
 \int_{N_i} (u, w_i)\,\nu_i(du) =0\qquad\text{for all} \quad w_i\in W_i,
\end{equation}
where $i = 1$ or $2$, as the case may be.
\end{theorem}
\begin{proof}[Proof. Discrete case]
The condition for the existence of a universal centering is given by
\eqref{12}, which by virtue of formula \eqref{16} is equivalent to
\begin{equation}\label{31}
 \int_{V_0}\frac{\|u\|^2-\|Au\|^2}{(1+\|A u\|^2)(1 +\|u\|^2)} u\,M(du)\,
 \perp\,\mathcal{N}(A^* - aI).
\end{equation}
For each $w\in W_1$ we have by \eqref{27}
\begin{align*}
 &\bigg(\int_{V_0}\frac{\|u\|^2-\|Au\|^2}{(1+\|A u\|^2)
 (1 +\|u\|^2)} u\,M(du),w\bigg)\\
 =&\bigg(\int_{V_0}\Big(\frac{1}{1+\|Au\|^2} - \frac{1}
 {1+\|u\|^2}\Big) u\,M(du), w\bigg)\\
 =&\sum_{n=-\infty}^\infty a^{-n}\bigg(\int_{Z_A}
 \Big(\frac{1}{1+\|A^{n+1}u\|^2} - \frac{1}{1+\|A^nu\|^2}\Big)
 A^n u\,M(du), w\bigg)\\
 =&\sum_{n=-\infty}^\infty a^{-n} \bigg(\int_{Z_A}
 \Big(\frac{1}{1+\|A^{n+1}u\|^2} -\frac{1}{1+\|A^n u\|^2}\Big)
 u\,M(du), A^{*n}w\bigg)\\
 =&\bigg(\int_{Z_A} \sum_{n=-\infty}^\infty \Big(\frac{1}
 {1+\|A^{n+1}u\|^2} - \frac{1}{1+\|A^n u\|^2}\Big) u\,M(du), w\bigg)\\
 =&\bigg(\int_{Z_A} \Big(\lim_{n\to\infty} \frac{1}{1+\|A^{n+1}u\|^2}
 - \lim_{n\to\infty} \frac{1}{1+\|A^{-n}u\|^2}\Big)u\,M(du), w\bigg)\\
 =&\bigg(\int_{Z_A}u\,M(du), w\bigg)
\end{align*}
since $\|A^n\| \leq \|A\|^n \to 0$ and
\[
 \|A^{-n} u\| \geq \frac{\|u\|}{\|A\|^n}\qquad\text{for}\quad n>0.
\]
Thus \eqref{31} is equivalent to \eqref{30} in the discrete case.

\emph{Continuous case}. By Proposition \ref{P12} and formula
\eqref{19} the existence of a universal centering is equivalent to
\begin{equation}\label{32}
 \int_{V_0} \frac{2(B u, u)}{(1+\|u\|^2)^2}\,u\,M(du)\,\perp\,
 \mathcal{N}(B^* - I).
\end{equation}
For each $w\in W_2$ we have by \eqref{29}
\begin{align*}
 &\bigg(\int_{V_0}\frac{2(B u, u)}{(1+\|u\|^2)^2}u\,M(du), w\bigg)\\
 =&\bigg(\int_{L_B} \int_{-\infty}^\infty \frac{2(Be^{tB}u, e^{tB}u)}
 {(1 +\|e^{tB}u\|^2)^2}\,e^{t(B-I)}u\,dt\,K_B(du), w\bigg)\\
 =&\int_{L_B} \int_{-\infty}^\infty \frac{2(Be^{tB} u, e^{tB} u)}
 {(1+\|e^{tB} u\|^2)^2}\,(u, e^{t(B^*-I)}w)\,dt\,K_B(du)\\
 =&\int_{L_B}\Big(\int_{-\infty}^\infty \frac{2(B e^{tB}u, e^{tB}u)}
 {(1 +\|e^{tB} u\|^2)^2}\,dt\Big) (u, w)\,K_B(du)\\
 =&c \int_{L_B} (u, w)\,K_B(du),
\end{align*}
where
\[
 c =\int_{-\infty}^\infty \frac{2(B e^{tB} u, e^{tB} u)}
 {(1+ \|e^{tB} u\|^2)^2}\,dt,
\]
and substitution
\[
 s = \|e^{tB} u\|^2
\]
gives in view of \eqref{18}
\[
 c = \int_0^\infty \frac{ds}{(1 +s)^2}=1.
\]
Consequently,
\[
 \bigg(\int_{V_0}\frac{2(B u, u)}{(1+\|u\|^2)^2}u\,M(du), w\bigg)=
 \int_{L_B} (u, w)\,K_B(du),
\]
thus \eqref{32} is equivalent to \eqref{30} in the continuous case
too, and the proof of the theorem has been finished.
\end{proof}

\begin{remark}
As noted in the Introduction, a condition equivalent to the existence
of centering in the continuous case was found in \cite{Sa}. However,
its form there is more complicated and does not fit into our
``homogeneous'' scheme given in Theorem \ref{T13}.
\end{remark}
The sets $Z_A, L_B$  and the measure $K_B$ depend on the choice of
operators $A$ and $B$. As for the continuous case, it is shown in
\cite{HJV} (cf. also \cite[Proposition 4.3.4]{JM}) that there is an
inner product on $V$ giving rise to a norm $||| \cdot |||$, and a
mixing measure $K$ on the unit sphere $L = \{v: |||v||| = 1\}$, such
that for every $B$ satisfying \eqref{13} we have
\[
 M(E) = \int_L \int_0^\infty \boldsymbol{1}_E(t^B u)\,\frac{dt}{t^2}\,K(du).
\]
In the discrete case, it can also be shown that under a suitable
inner product norm we can have the set $Z_A$ independent of $A\in
G_a(\mu)$ (though it will still depend on $a$) (cf. \cite{L3}).
Denoting this set by $Z$, we get the formula
\[
 M(E) =\sum_{n=-\infty}^\infty a^{-n} \int_Z \boldsymbol{1}_E(A^n u)\,M(du)
\]
for every $A$ satisfying \eqref{2}. Theorem \ref{T13} may now be
given the following form.
\begin{Theorem}
The existence of a universal quasi-de\-com\-po\-sa\-bi\-li\-ty
centering is equivalent to the conditions:
\begin{enumerate}
 \item[(i)] Discrete case
 \[
  \int_Z (u,w)\,M(du)=0 \qquad\text{for all }w\in\mathcal{N}(A^*-aI)
 \]
 where $A$ is \emph{any} operator satisfying \eqref{2};
 \item[(ii)] Continuous case
 \[
  \int_L (u,w)\,K(du)=0 \qquad\text{for all }
  w\in\mathcal{N}(B^*-aI),
 \]
 where $B$ is \emph{any} operator satisfying \eqref{13}.
\end{enumerate}
\end{Theorem}

\begin{remark}[cf. \cite{JM,Sa}]
Ordinary multivariate semistable and stable measures are obtained if
there is an operator $A$ satisfying \eqref{2}, or operator $B$
satisfying \eqref{13}, respectively, being a multiple of the
identity. The only problem with the existence of centering in this
case arises when $A=aI$ and $B=I$. In such a case we have
\begin{align*}
 &Z_A=\{v\colon a<\|v\|\leq 1\}, \qquad L_B=\{v\colon \|v\|=1\},\\
 &\mathcal{N}(A^*-aI)=\mathcal{N}(B^*-I)=V,
\end{align*}
thus the conditions on the existence of centering are respectively:
\[
 \int_{a<\|u\|\leq 1}u\,M(du)=0,\qquad \int_{\|u\|=1}u\,M(du)=0.
\]
\end{remark}

\emph{Acknowledgments}. The author is indebted to Dr. Hans-Peter
Scheffler for calling his attention to paper \cite{Sa}.

\end{document}